%% file: ERDOSGALLAI_NOTEC.tex
\documentclass[a4paper,12pt,reqno]{amsart}
\usepackage[utf8]{inputenc}
\usepackage[english]{babel}
\usepackage{amsthm,amsmath,amsfonts,amssymb}
\usepackage{array,enumerate,float,moreverb,appendix,mathrsfs}
\usepackage[svgnames]{xcolor}
\usepackage[final]{graphicx}
\usepackage[colorlinks=true,citecolor=MediumVioletRed,linkcolor=DarkCyan]{hyperref}
\usepackage[lofdepth,lotdepth]{subfig}
\usepackage[left=2.5 cm,right=2.5cm,top=2.5cm,bottom=2.5cm]{geometry}
\numberwithin{equation}{section}
\usepackage{transparent}

\theoremstyle{plain}
\newtheorem{theorem}{Theorem}
\newtheoremstyle{thm}
  {10pt}
  {6pt}
  { \slshape}
  {}
  {\bfseries \scshape}
  {. }
  { }
  {}
\theoremstyle{thm} 
\newtheorem{prop}{Proposition}
\newtheorem{lem}[prop]{Lemma}

\newtheorem{question}[prop]{Question}

\numberwithin{prop}{section}
\usepackage{etoolbox}
\patchcmd{\section}{\scshape}{\scshape\large}{}{}
\patchcmd{\section}{.7}{1.4}{}{}
\patchcmd{\section}{.5}{1.2}{}{}

\title{Graphs with prescribed local neighborhoods of their universal coverings.}
\author{Charles Bordenave and Simon Coste}
\date{\today}
\begin{document}
\bibliographystyle{alpha}

\begin{abstract}
Given a collection of $n$ rooted trees with depth $h$, we give a necessary and sufficient condition for this collection to be the collection of $h$-depth universal covering neighborhoods at each vertex. 
\end{abstract}

\maketitle

\section{Reconstruction of a graph with its universal covering.}

Let $G=(V,E)$ be a finite, connected graph. A graph $G' = (V',E')$ is a \emph{covering} of $G$ if there is a surjective map $\iota : V' \to V$ which is a local isomorphism:  for every $x \in V'$,  $\iota$ induces a bijection between the edges incident to $x$ and the edges incident to $\iota(x)$.  The \emph{universal covering} of $G$, denoted by $T_G$, is the unique (up to isomorphism) covering  which is a tree. Note that $T_G$ can be infinite if $G$ is not itself a tree. For instance, the universal covering of any $d$-regular graph is the infinite $d$-regular tree $\mathbb{T}_d$. Note that $T_G$ is also the universal covering of any covering of $G$.

Let $h$ be a positive integer. Given any vertex $x$ of $G$, its $h$-depth universal covering neighborhood is the unlabeled ball of radius $h$ in $T_G$ around any antecedent of $x$ by $\iota$. One can easily see that this ball does not depend --- up to isomorphism --- on the chosen antecedent $x$. 

\begin{figure}[H]\centering
\includegraphics[scale=0.9]{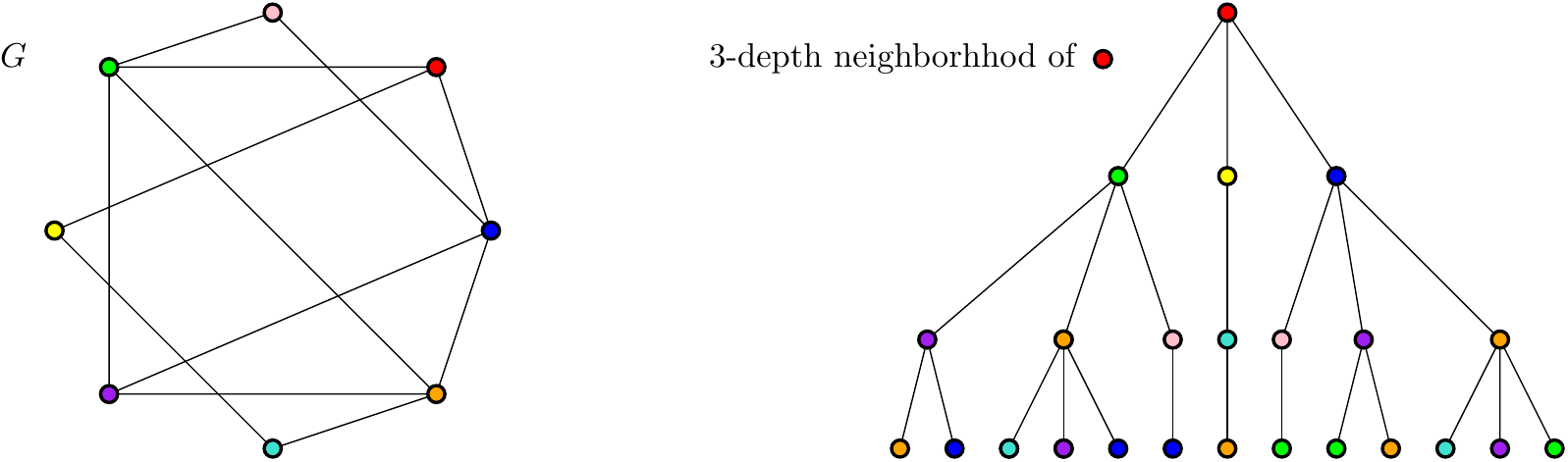}
\caption{The graph $G$ is on the left ; the $3$-depth neighborhood of the red vertex is drawn on the right. Two vertex having the same color in $T_G$ are antecedents of the vertex with the corresponding color in $G$. }
\end{figure}

\begin{question}\label{question}Let $\mathbf{t}=(t_1, ..., t_n)$ be a collection of $n$ unlabeled rooted trees with maximal depth $h$. Is it the collection of $h$-depth universal covering neighborhoods of some (simple) graph $G$ ? 
\end{question}

If this is the case, we call the $n$-tuple $\mathbf{t}$ a \emph{graphical $h$-neighborhood} and we say that $G$ is a \emph{realization} of $\mathbf{t}$.


\subsection{Notation} From now on, we will adopt the term ``tree" instead of ``unlabeled, rooted tree", unless explicitly stated otherwise. If $G$ is any graph and $x$ is a vertex of $G$, we will note $\mathrm{deg}_G(x)$ the number of neighbors of $x$ in $G$. In a directed graph $G$, we will note $\mathrm{deg}^\pm_G$ the in and out degree of the vertices. Generally, the root of a (rooted) graph will be noted $\bullet$ and if $k$ is an integer and $g$ a rooted graph, $(g)_k$ denotes the ball $B_g(\bullet, k)$ of radius $k$ around the root of $g$. The set of all (unlabeled, rooted) trees with depth $h$ will be noted $\mathscr{T}_h$.  

\subsection{Related work.}\emph{Graph reconstruction problems} ask the following question : given any property $\mathscr{P}$ about graphs, how can we ensure that there is a graph (or digraph, or multigraph) having this property $\mathscr{P}$ ? What is the number of graphs that have this property ? Can we determine the properties $\mathscr{P}$ that have a single graph realization ?

Reconstructing a graph (or bipartite graph, or digraph) only by the list of its degree has been a well-known and studied problem since the seminal works of Erdös, Gallai and many others. In fact, question \ref{question} had been settled long time ago for $h=1$ by the celebrated Erdös-Gallai theorem. Suppose that $\mathbf{t}=(t_1, ..., t_n)$ is an $n$-tuple of 1-depth trees. A $1$-depth tree $t_i$ is just a root with some leaves, say $d_i$ leaves ; thus, a $1$-depth neighborhood can be identified with a $n$-tuple of integers $(d_1, ..., d_n)$.  

\begin{figure}[H]\centering
\includegraphics[scale=0.8]{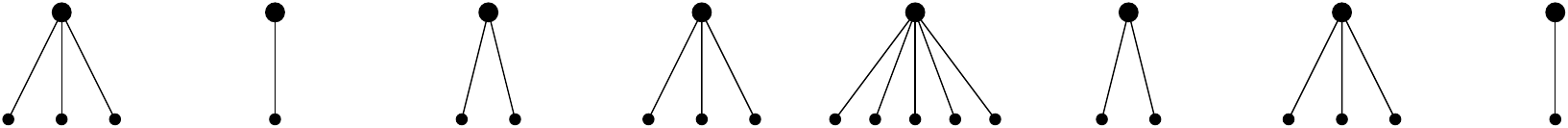}
\caption{Is there a graph on $8$ vertices with this $1$-depth neighborhood ? The associated degree sequence if $\mathbf{d}=(3, 1, 2, 3, 5, 2, 3, 1)$. According to Theorem \ref{thm:EG_simple}, yes.}
\end{figure}

Finding a graph $G$ with $\mathbf{t}$ as $1$-depth neighborhood boils down to finding a graph $G$ with degree sequence $\mathbf{d}$ --- such sequences are called \emph{graphical}. All integer sequences are not graphical ; the Erdös-Gallai theorem gives one necessary and sufficient condition for an integer sequence to be graphical.

\begin{theorem}[Erdös, Gallai, \cite{ERDOSGALLAI}]\label{thm:EG_simple}
Let $\mathbf{d}=(d_1, ..., d_n)$ be a $n$-tuple of integers. Rearrange them in decreasing order $d_{(1)}\geqslant ... \geqslant d_{(n)}$. Then, $\mathbf{d}$ is graphical if and only if it satisfies the two following conditions : 
\begin{equation}\label{EG:even}
d_1+\cdots +d_n \text{ is even}, 
\end{equation}
and the ``Erdös-Gallai condition"
\begin{equation}\label{EG:comp}
\forall k \in [n], \quad d_1+\cdots +d_k \leqslant k(k-1) + \sum_{i=k+1}^n d_i \wedge k.
\end{equation}
\end{theorem}

A short and constructive proof is available at \cite{tripathi2003note}. In fact, the Erdös-Gallai condition is not the only sufficient and necessary condition for an integer sequence to be graphical ; there are some other (equivalent) conditions, notably listed in \cite{EGSURV}. The corresponding realization problem for \emph{digraphs} had also been solved quite early ; see the interesting note \cite{annabell} for a complete history and presentation of the many variants.

\begin{theorem}\label{thm:EG_directed}
Let $\mathbf{d}^\pm=(d_i^+, d_i^-)_{i \in [n]}$ be a $2n$-tuple of integers. We order the first component by decreasing order $d_{(1)}^+ \geqslant ... \geqslant d_{(n)}^+$. Then, $\mathbf{d}^\pm$ is the sequence of oriented in and out degrees of some digraph $G$ if and only if it satisfies the two following conditions : 
\begin{equation}
\sum_{i=1}^n d_i^+ = \sum_{i=1}^n d_i^-, 
\end{equation}
and the ``directed Erdös-Gallai condition" : 
\begin{equation}
\forall k \in [n], \quad \sum_{i=1}^k d^+_{(i)} \leqslant \sum_{i=1}^k d^-_{(i)}\wedge (k-1) + \sum_{i=k+1}^n k \wedge d_{(i)}^-
\end{equation}
where the couples $(d^+_{(i)}, d^-_{(i)})$ are sorted in lexicographic order. 
\end{theorem}

This settled our question for $h=1$. The case $h=2$ had recently been solved by \cite{NDL_physique, NDL} ; a $2$-depth neighborhood is called a \emph{neighborhood degree list} (NDL). In \cite{NDL}, the authors not only settle Question \ref{question} and give a sufficient and necessary condition for a NDL to be graphical, but they also characterize those NDL that are ``unigraphical", meaning that they have a unique graphical realization --- we do not adress this problem, but we solve Question \ref{question} for arbitrary depths $h$.

For $h=1$, the number of labeled graphs with a given degree sequence is asymptotically known in many asymptotic regimes, see notably \cite{MR3317354}, \cite[Theorem 2.16]{MR1864966} and references therein. For general $h$, this question has been recently adressed in  \cite{bordenave2015large} in the regime where the maximal degree is uniformly bounded. The motivation came from the Benjamini-Schramm topology of rooted graphs.

In this paper, we only deal with \emph{universal covering neighborhoods}, thus ignoring the eventual cycles in the $h$-neighborhood of a vertex. If a $h$-depth neighborhood is graphical, then it might as well have very different realizations, for instance ones that are $h$-locally tree-like, or others having many short cycles. When the same question is adressed with \emph{graph} $h$-neighborhoods, Question \ref{question} becomes much more arduous ; a similar problem in graph reconstruction, the famous Kelly-Ulam reconstruction problem, was asked during the 1940s and still remains opened.

\subsection{Definitions and statement of the main result.}

Fix some $h$-depth neighborhood $\mathbf{t}=(t_1, ..., t_n)$. The associated degree sequence $\mathbf{d}=(d_1, ..., d_n)$ is the sequence of degrees of the root $\bullet$ of the tree $t_i$, that is $d_i = \mathrm{deg}_{t_i}(\bullet)$. An obvious necessary condition for $\mathbf{t}$ to be graphical is that $\mathbf{d}$ is itself a graphical sequence, hence satisfying \eqref{EG:even}-\eqref{EG:comp}. From now on, we will assume that $d_1+...+d_n = 2m$ where $m$ is an integer. 

Let $t$ be a tree with depth at most $h$ and root $\bullet$. Let $e$ be an edge incident with the root, say $e=(\bullet, x)$. The tree $t \setminus e$ has exactly two connected components. The connected component containing the root is $r'$ and the other one is $s$ ; we root $s$ at $x$. We erase from $r'$ all the vertices which were at depth exactly $h$ in $t$, and we keep the same root ; this yields a new rooted tree $r$ --- see Figure \ref{fig2}. The \emph{type} of the edge $e$ is defined as the couple of rooted trees $(r,s)$ and we will denote it by $\tau(e)$. 

\begin{figure}[H]\centering
\begin{tabular}{ccccc}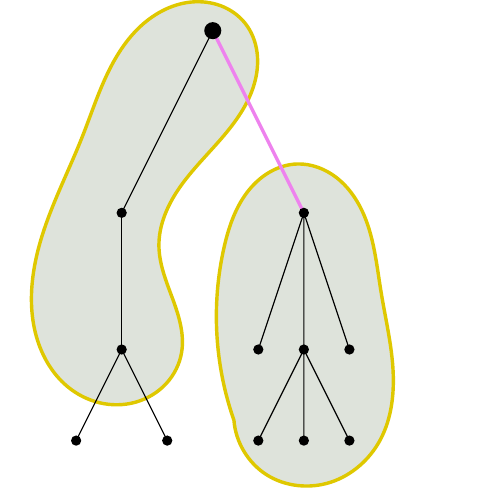&&&&\includegraphics[scale=1.2]{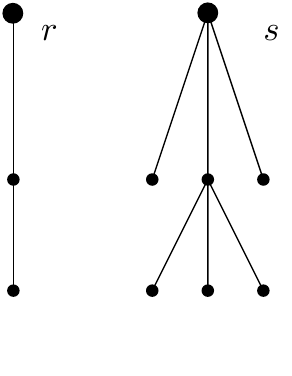}
\end{tabular}
\caption{Construction of the type $\tau(e) = (r,s)$ of edge $e$ in some tree $t$.}\label{fig2}
\end{figure}

If $\tau=(r,s)$ is a type, its \textit{opposite type} $\tau^{-1}$ is defined as $(s,r)$. A type is an element of $\mathscr{T}_{h-1} \times \mathscr{T}_{h-1}$. The set of all types induced by the edges in $\mathbf{t}$ is noted $\mathrm{types}(\mathbf{t})$. It can be decomposed into the disjoint union of three sets 
$$\mathrm{types}(\mathbf{t}) = \Delta \cup A \cup B$$
where 
\begin{itemize}
\item $\Delta$ is the set of ``diagonal" types $\tau = (r,r)$ for some $r \in \mathscr{T}_{h-1}$ ; 
\item $A \cup B$ is the set of types $\tau = (r,s)$ with $r \neq s$, and the sets $A,B$ are chosen such that if $\tau \in A$, then $\tau^{-1} \in B$. 
\end{itemize}

If $\tau \in \mathrm{types}(\mathbf{t})$, we define
\begin{itemize}
\item the $\tau$-degree of any index $i \in [n]$ as the number of edges in $t_i$ incident to the root and whose type is $\tau$. We will denote it by $d^\tau_i$ ; 
\item the $\tau$-number $N_\tau$ as the total number of edges in $\mathbf{t}$ with type $\tau$, that is
$$N_\tau = \sum_{i \in [n]} d^\tau_i.$$
\end{itemize}
It should be clear that if $ i \in [n]$ is a vertex, then $\sum_{\tau \in \mathrm{types}(\mathbf{t})} d^\tau_i = d_i$. 

\begin{theorem}\label{thm:main0}
Let $\mathbf{t}$ be a $h$-depth neighborhood ; it is graphical if and only if it satisfies the following conditions : 
\begin{itemize}
\item for every $\tau \in \Delta$, the integer sequence $(d^\tau_i)_{i \in [n]}$ is graphical ; 
\item for every $\tau\in A$, the integer double sequence $(d^\tau_i, d^{\tau^{-1}}_i)_{i \in [n]}$ is digraphical.
\end{itemize}
\end{theorem}

Using classical characterizations of graphical and digraphical sequences given earlier in Theorems \ref{thm:EG_simple} and \ref{thm:EG_directed}, this result can be detailed : 

\begin{theorem}\label{thm:main}
Let $\mathbf{t}$ be a $h$-depth neighborhood ; it is graphical if and only if it satisfies the following conditions : 
\begin{itemize}
\item for every $\tau \in \Delta$, the integer $N_\tau$ is even and for every $k \in [n]$ we have
\begin{equation}\label{EG-sym}
\sum_{i=1}^k d^\tau_{(i)} \leqslant k(k-1) + \sum_{i=k+1}^n d_{(i)} \wedge k, 
\end{equation}
\item for every $\tau \in A$, we have $N_\tau = N_{\tau^{-1}}$ and for every $k \in [n]$, we have 
\begin{equation}\label{EG-nonsym}
\sum_{i=1}^k d^\tau_{(i)} \leqslant \sum_{i=1}^k d^{\tau^{-1}}_{(i)} + \sum_{i=k+1}^n d^{\tau^{-1}}_{(i)} \wedge k
\end{equation}
where indices correspond to lexicographic reordering. 
\end{itemize}
\end{theorem}

Note that those conditions together imply that $(d_1, ..., d_n)$ is itself a graphical sequence (sum over all the types $\tau$), which is a necessary, but clearly non sufficient condition.

\section{Proof of Theorem \ref{thm:main0}.}

We assume without loss of generality $h \geq 2$. The conditions are easily seen to be necessary, for if $\mathbf{t}$ is graphical and $\tau$ is a type, then
\begin{itemize}
\item either $\tau \in \Delta$ and the graph induced in $G$ by keeping only the edges $e$ such that $\tau(e) = \tau$ has $(d^\tau_i)_{i \in [n]}$ has its degree sequence, 
\item either $\tau \notin \Delta$ ; in this case either $\tau \in A$ or $\tau^{-1} \in A$, so without loss of generality we can assume that $\tau \in A$. The graph induced by edges such that $\tau(e) = \tau$ can be oriented : if $e=(i,j) \in G_\tau$, then one vertex $k \in \{i,j\}$ satisfies $\tau(e) = \tau$ in $t_k$. We orient the edge $(i,j)$ from $k$ to the other vertex. This yields a digraph $\vec{G}_\tau$ with oriented bi-degree sequence $(d^\tau_i, d^{\tau^{-1}}_i)_{i \in [n]}$, so the second condition of Theorem \ref{thm:main0} is met. 
\end{itemize}

\bigskip

We now prove the sufficiency. We suppose that $\mathbf{t}$ is a $h$-depth neighborhood satisfying the assumptions in Theorem \ref{thm:main0} and we build a graph $G$ which is a realization of $\mathbf{t}$. We first fix some type $\tau$. 

\begin{itemize}\item We suppose in the first time that $\tau \in A$, in particular $\tau=(r,s)$ with $r \neq s$. As $(d^\tau_i, d^{\tau^{-1}}_i)_{i \in [n]}$ is digraphical, there is some digraph $\vec{G}_\tau$ on $n$ vertices such that $\mathrm{deg}^+_{\vec{G}_\tau}(i) = d^\tau_i$ and $\mathrm{deg}^-_{\vec{G}_\tau}(i) = d^{\tau^{-1}}_i$ for every vertex $i \in [n]$. We now define a (non-directed) multigraph $G_\tau$ by simply forgetting the directions of edges in $\vec{G}_\tau$ --- indeed, this multigraph will be proven to be simple in Lemma \ref{simple2}.
\item Else, if $\tau \in \Delta$, then by assumption $(d^\tau_i)_{i \in [n]}$ is graphical and there is a simple graph $G_\tau$ such that $\mathrm{deg}_{G_\tau}(i) = d^\tau_i$. 
\end{itemize}

We now ``glue together" the graphs $G_\tau$ to get our realization of $\mathbf{t}$, namely $G$. Formally, if $E(G_\tau)$ denotes the set of edges in $G_\tau$, then $G = ([n], E)$ with the edge set $E$ being defined as 
\begin{equation}
E:= \bigcup_{\tau \in \Delta \cup A} E\big( G_\tau \big).
\end{equation}

The following lemma is the crucial ingredient of the proof of Theorem \ref{thm:main0}. 

\begin{lem}\label{simple2}$G$ is a simple graph.
\end{lem}

\begin{proof}
Suppose that $G_\tau$ contains a double edge, for instance $(x,y)$. We are going to prove the two following facts : 
\begin{enumerate}
\item first, this double edge can not arise from two distinct $G_\tau$. In other words, if $(x,y) \in G_\tau$, then $(x,y) \notin G_{\tau'}$ for every $\tau'\neq \tau$ ; 
\item then we check that for every $\tau \in  A$, the multi-graph $G_\tau$ contains no double edge.
\end{enumerate}

Together, those two facts imply that $G$ is simple : indeed, if there is a double edge, then it can only belong to a single $G_\tau$ ; but if $\tau \in A$, $G_\tau$ cannot contain any double edge, and if $\tau \in \Delta$ then $G_\tau$ is simple by construction, hence the conclusion.

\bigskip

Suppose that there is a double edge between vertices  $i$ and $j$, one belonging to $G(\tau)$ and the other to $G(\tau')$ for two types $\tau=(r,s)$ and $\tau'=(r',s')$. We prove that $\tau=\tau'$. As manipulating unlabeled rooted trees is quite inconvenient, we will work with two labeled rooted trees $T_i, T_j$ in the equivalence classes of $t_i, t_j$, and the same with $R,R',S,S'$ which are representatives of the equivalence classes of $r,r',s,s'$. We are going to prove that $R \simeq R'$ and $S \simeq S'$ (as rooted labeled trees) , hence proving $r=r'$ and $s=s'$ as needed. The following arguments are illustrated in Figure \ref{fig_simple}. 

\begin{itemize}
\item The presence of an edge between $i$ and $j$ in $G_\tau$ has the following consequence : there is an edge $e$ in $T_i$, adjacent with the root, such that $T_i \setminus e$ has two connected components, one isomorphic with $S$ and the other having its ball of radius $h$ isomorphic with $R$. On the other hand, as $(i,j) \in G(\tau')$, there is an edge $e'$ such that $T_i \setminus e'$ has one component isomorphic with $S'$ and the ball of radius $h-1$ of the other is isomorphic with $R'$.
\item The same holds with $T_j$. 
\end{itemize}

It is clear that $\mathrm{deg}_R(\bullet)+1 = \mathrm{deg}_{T_i}(\bullet)=d_i$ and also $\mathrm{deg}_{R'}(\bullet)+1 =d_i$, hence $\mathrm{deg}_R(\bullet)=\mathrm{deg}_{R'}(\bullet)$. The same is true with $S,S'$ ; we have just proven that $(R)_1 \simeq (R')_1$ and $(S)_1 \simeq (S')_1$. We are now going to prove that if $(S)_k\simeq (S')_k$ and $(R)_k \simeq (S')_k$ for some $k<h-1$, then this is also true with $k+1$. 

First, the ball $(T_i)_{k+1}$ can be decomposed in two ways shown in Figure \ref{fig_simple} : 
$$(T_i)_{k+1} = e \cup (S)_k \cup (R)_{k+1} \qquad \text{and} \qquad (T_i)_{k+1} =e' \cup (S')_k \cup (R')_{k+1} $$
but as $(S)_k \simeq (S')_k$, we can erase both branches pending at $e$ and $e'$, to get $(R)_{k+1} \simeq (R')_{k+1}$. The same idea applies to $T_j$, to show that $(S)_{k+1} \simeq (S')_{k+1}$, hence closing the recurrence. We have proven that $(S)_{h-1} \simeq (S')_{h-1}$ and $(R)_{h-1} \simeq (R')_{h-1}$, thus $r=r'$ and $s=s'$ as needed.  We thus have proven the first point exposed earlier.

\begin{figure}[H]\centering
\begin{tabular}{cc}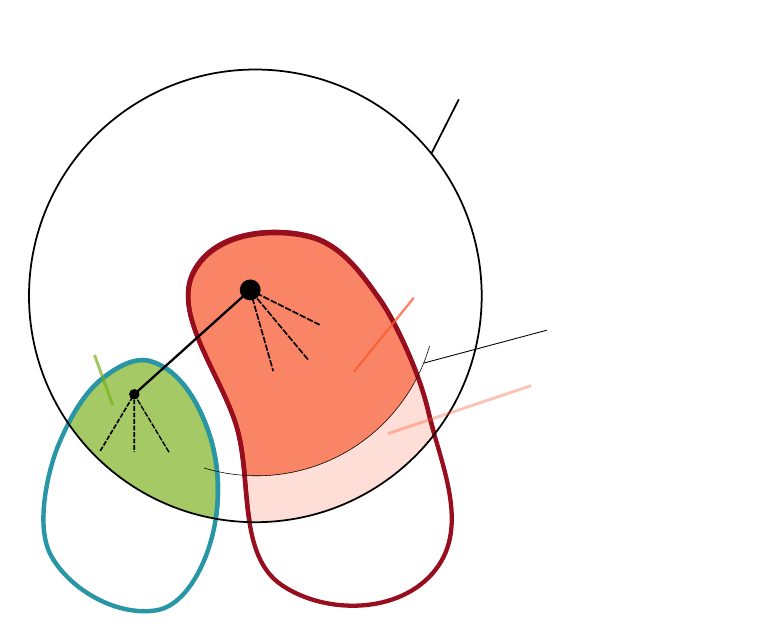&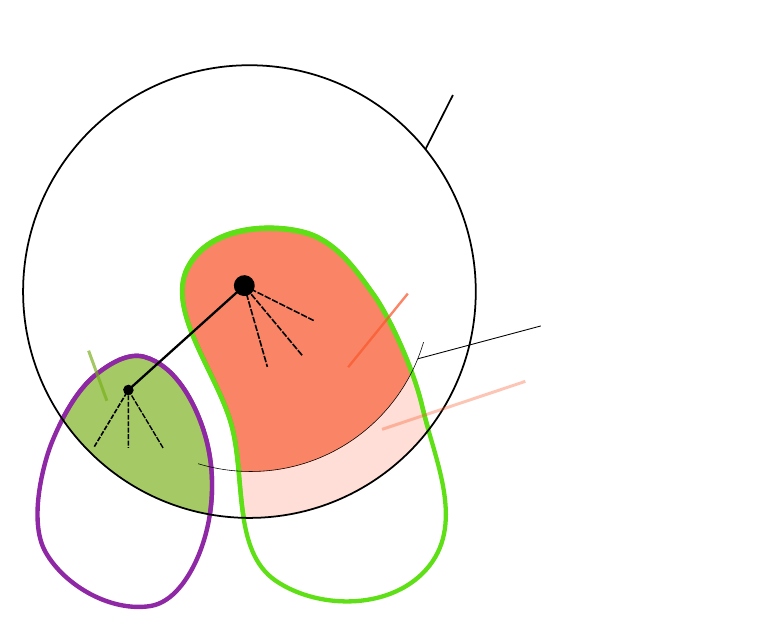 
\end{tabular}
\caption{An illustration of the proof of Lemma \ref{simple2}. The green parts represent $(S)_k$ and $(S')_k$ which are isomorphic (as recurrence hypothesis) and the dark red parts are representing $(R)_k$ and $(R')_k$, which are isomorphic too ; hence, the light pink parts are also isomorphic, thus proving $(R)_{k+1} \simeq (R')_{k+1}$. A similar procedure applies to $T_j$.}\label{fig_simple}
\end{figure}

We now check the second point, i.e. that for every $\tau \in A$, the multi-graph $G_\tau$ is indeed a simple graph. The proof runs along the same lines : suppose that there is a double directed edge between $i$ and $j$ in $\vec{G}_\tau$. This can only happen if $(i,j)$ and $(j,i)$ are both directed edges in $\vec{G}_\tau$. We suppose that $\tau = (r,s)$, and with a recurrence we prove that $r=s$, hence $\tau \in \Delta$ which had been discarded since $\Delta \cap A = \emptyset$. 

To do this, first check that $\mathrm{deg}_r(\bullet)=\mathrm{deg}_s(\bullet)$, then suppose that for some $k<h$, we have $(r)_k = (s)_k$ and prove that $(r)_{k+1} = (s)_{k+1}$. This step uses the exact same procedure as before.
\end{proof}

We now check that $G$ solves our problem. 

\begin{lem}$G$ is a realization of $\mathbf{t}$.
\end{lem}

\begin{proof}
We want to show that the $h$-neighborhood of any vertex $i$ in the universal cover of $G$ matches $t_i$. We show by strong recurrence that for $k \leqslant h$, if $\bar{t}_i$ denotes the $h$-neighborhood of $i$ in the universal cover, then for every $i \in [n]$ we have $(t_i)_k =(\bar{t}_i)_k$. It is clear by our construction of $G$ that $\mathrm{deg}_G(i) = \sum_\tau d^\tau_i = d_i$, hence $(t_i)_1 = (\bar{t}_i)_1$. Now suppose that $(t_i)_k =(\bar{t}_i)_k$ for some $k<h$. If $N_G(i)$ is the set of neighbors of $i$ in $G$, then for every $j \in N_G(i)$ we have $(t_j)_k=(\bar{t}_j)_k$ by the recurrence hypothesis. This readily implies that $(t_i)_{k+1} = (\bar{t}_i)_{k+1}$, hence the lemma is proven.
\end{proof}

\bibliography{bibli}

\end{document}

%% file: type_e.pdf_tex
\begingroup%
  \makeatletter%
  \providecommand\color[2][]{%
    \errmessage{(Inkscape) Color is used for the text in Inkscape, but the package 'color.sty' is not loaded}%
    \renewcommand\color[2][]{}%
  }%
  \providecommand\transparent[1]{%
    \errmessage{(Inkscape) Transparency is used (non-zero) for the text in Inkscape, but the package 'transparent.sty' is not loaded}%
    \renewcommand\transparent[1]{}%
  }%
  \providecommand\rotatebox[2]{#2}%
  \ifx\svgwidth\undefined%
    \setlength{\unitlength}{143.83066132bp}%
    \ifx\svgscale\undefined%
      \relax%
    \else%
      \setlength{\unitlength}{\unitlength * \real{\svgscale}}%
    \fi%
  \else%
    \setlength{\unitlength}{\svgwidth}%
  \fi%
  \global\let\svgwidth\undefined%
  \global\let\svgscale\undefined%
  \makeatother%
  \begin{picture}(1,0.9766507)%
    \put(0,0){\includegraphics[width=\unitlength,page=1]{type_e.pdf}}%
    \put(0.54168586,0.72325235){\color[rgb]{0.93333333,0.50980392,0.93333333}\makebox(0,0)[lb]{\smash{$e$}}}%
    \put(-0.00946361,0.39262792){\color[rgb]{0,0,0}\makebox(0,0)[lb]{\smash{$r$}}}%
    \put(0.80488294,0.3502816){\color[rgb]{0,0,0}\makebox(0,0)[lb]{\smash{$s$}}}%
  \end{picture}%
\endgroup%

%% file: P1.pdf_tex
\begingroup%
  \makeatletter%
  \providecommand\color[2][]{%
    \errmessage{(Inkscape) Color is used for the text in Inkscape, but the package 'color.sty' is not loaded}%
    \renewcommand\color[2][]{}%
  }%
  \providecommand\transparent[1]{%
    \errmessage{(Inkscape) Transparency is used (non-zero) for the text in Inkscape, but the package 'transparent.sty' is not loaded}%
    \renewcommand\transparent[1]{}%
  }%
  \providecommand\rotatebox[2]{#2}%
  \ifx\svgwidth\undefined%
    \setlength{\unitlength}{224.76742722bp}%
    \ifx\svgscale\undefined%
      \relax%
    \else%
      \setlength{\unitlength}{\unitlength * \real{\svgscale}}%
    \fi%
  \else%
    \setlength{\unitlength}{\svgwidth}%
  \fi%
  \global\let\svgwidth\undefined%
  \global\let\svgscale\undefined%
  \makeatother%
  \begin{picture}(1,0.81949596)%
    \put(0,0){\includegraphics[width=\unitlength,page=1]{P1.pdf}}%
    \put(0.42333626,0.1098045){\color[rgb]{0.59215686,0.05490196,0.12156863}\makebox(0,0)[lb]{\smash{$R$}}}%
    \put(0.08292205,0.11952006){\color[rgb]{0.15686275,0.58823529,0.64705882}\makebox(0,0)[lb]{\smash{$S$}}}%
    \put(0.31549443,0.47365086){\color[rgb]{0,0,0}\makebox(0,0)[lb]{\smash{$i$}}}%
    \put(0.69622581,0.31717336){\color[rgb]{0.98039216,0.34117647,0.16078431}\transparent{0.34033617}\makebox(0,0)[lb]{\smash{$(R)_{k+1}$}}}%
    \put(0.0763634,0.38064088){\color[rgb]{0.46666667,0.68235294,0.09411765}\transparent{0.67058825}\makebox(0,0)[lb]{\smash{$(S)_k$}}}%
    \put(0.52727513,0.45685616){\color[rgb]{0.97647059,0.34117647,0.17254902}\transparent{0.83613443}\makebox(0,0)[lb]{\smash{$(R)_k$}}}%
    \put(0.52155264,0.71801087){\color[rgb]{0,0,0}\makebox(0,0)[lb]{\smash{$(T_i)_{k+1}$}}}%
    \put(0.71333497,0.39696787){\color[rgb]{0,0,0}\makebox(0,0)[lb]{\smash{$(T_i)_{k}$}}}%
    \put(0.20359572,0.3694242){\color[rgb]{0,0,0}\makebox(0,0)[lb]{\smash{$e$}}}%
  \end{picture}%
\endgroup%

%% file: P2.pdf_tex
\begingroup%
  \makeatletter%
  \providecommand\color[2][]{%
    \errmessage{(Inkscape) Color is used for the text in Inkscape, but the package 'color.sty' is not loaded}%
    \renewcommand\color[2][]{}%
  }%
  \providecommand\transparent[1]{%
    \errmessage{(Inkscape) Transparency is used (non-zero) for the text in Inkscape, but the package 'transparent.sty' is not loaded}%
    \renewcommand\transparent[1]{}%
  }%
  \providecommand\rotatebox[2]{#2}%
  \ifx\svgwidth\undefined%
    \setlength{\unitlength}{224.1960084bp}%
    \ifx\svgscale\undefined%
      \relax%
    \else%
      \setlength{\unitlength}{\unitlength * \real{\svgscale}}%
    \fi%
  \else%
    \setlength{\unitlength}{\svgwidth}%
  \fi%
  \global\let\svgwidth\undefined%
  \global\let\svgscale\undefined%
  \makeatother%
  \begin{picture}(1,0.82158465)%
    \put(0,0){\includegraphics[width=\unitlength,page=1]{P2.pdf}}%
    \put(0.41682929,0.11558859){\color[rgb]{0.37254902,0.87843137,0.08627451}\makebox(0,0)[lb]{\smash{$R'$}}}%
    \put(0.07554751,0.12532913){\color[rgb]{0.56078431,0.15686275,0.64705882}\makebox(0,0)[lb]{\smash{$S'$}}}%
    \put(0.30871249,0.48036241){\color[rgb]{0,0,0}\makebox(0,0)[lb]{\smash{$i$}}}%
    \put(0.69041431,0.32348609){\color[rgb]{0.98039216,0.34117647,0.16078431}\transparent{0.34033617}\makebox(0,0)[lb]{\smash{$(R')_{k+1}$}}}%
    \put(0.06897193,0.38711526){\color[rgb]{0.46666667,0.68235294,0.09411765}\transparent{0.67058825}\makebox(0,0)[lb]{\smash{$(S')_k$}}}%
    \put(0.52103302,0.46352479){\color[rgb]{0.97647059,0.34117647,0.17254902}\transparent{0.83613443}\makebox(0,0)[lb]{\smash{$(R')_k$}}}%
    \put(0.51529594,0.72534523){\color[rgb]{0,0,0}\makebox(0,0)[lb]{\smash{$(T_i)_{k+1}$}}}%
    \put(0.70756697,0.40348419){\color[rgb]{0,0,0}\makebox(0,0)[lb]{\smash{$(T_i)_{k}$}}}%
    \put(-1.71234986,0.89251354){\color[rgb]{0,0,0}\makebox(0,0)[lt]{\begin{minipage}{2.06961758\unitlength}\raggedright \end{minipage}}}%
    \put(0.19790857,0.38041532){\color[rgb]{0,0,0}\makebox(0,0)[lb]{\smash{$e'$}}}%
  \end{picture}%
\endgroup%